\documentclass{PRM}

\newtheorem{theorem}{Theorem}

\usepackage[utf8]{inputenc} 
\usepackage[T1]{fontenc}    
\usepackage{hyperref}       
\usepackage{amsfonts}       
\usepackage{amsmath}

\usepackage{amsfonts,amssymb}
\usepackage{amsthm}

\newcommand{\R}{\mathbb{R}}

\newcommand{\bb}{\begin{equation}}

\usepackage[all]{xy}

\numberwithin{equation}{section}

\newcommand{\sobolev}[2]{\Vert #1\Vert_{H^{#2}}}
\newcommand{\hm}[3]{\vert\vert\vert #1\vert\vert\vert_{E_{#2,#3}}}
\newcommand{\km}[3]{\Vert #1\Vert_{#2,#3}}
\newcommand{\kma}[2]{\Vert #1\Vert_{#2,2,m}}

\newtheorem{lemma}[theorem]{Lemma}
\newtheorem{proposition}[theorem]{Proposition}
\newtheorem{corollary}[theorem]{Corollary}

\begin{document}

\title[Well-posedness and global analyticity for a generalized $0$-equation]{Local well-posedness and global analyticity for solutions of a generalized $0$-equation}

\author{\name{Priscila L. \surname{da Silva}}}

\address{Department of Mathematical Sciences, School of Science, Loughborough University, Loughborough, UK \email{P.Leal-Da-Silva@lboro.ac.uk}}

\address{Centre of Mathematics, Computation and Cognition, Universidade Federal do ABC, Brazil \email{priscila.silva@ufabc.edu.br}}



\begin{abstract}
In this work we study the Cauchy problem in Gevrey spaces for a generalized class of equations that contains the case $b=0$ of the $b$-equation. For the generalized equation, we prove that it is locally well-posed for initial data in Gevrey spaces. Moreover, as we move to global well-posedness, we show that for a particular choice of the parameter in the equation the local solution is global analytic in both time and spatial variables.
\end{abstract}



\maketitle

\section{Introduction}

The 4-parameter equation
\begin{align}\label{4par}
    u_t - u_{txx} + au^ku_x - bu^{k-1}u_xu_{xx} - cu^ku_{xxx}=0,\quad a,b,c\in \R\setminus\{0\},\quad k\in\mathbb{N},
\end{align}
studied in \cite{Eu,Eu2,HH}, is a generalization of the Camassa-Holm equation \cite{CH}
\begin{align}\label{CH}
    u_t - u_{txx} + 3uu_x - 2u_xu_{xx} -uu_{xxx}=0,
\end{align}
and the Novikov equation \cite{HW,Novikov}
\begin{align}\label{Nov}
    u_t - u_{txx} + 4u^2u_x - 3uu_xu_{xx} - u^2u_{xxx}=0,
\end{align}
that admits certain scaling transformations as symmetries. The equation \eqref{4par} has proven to be an interesting mathematical equation once it is possible to choose $a=k+2,b=k+1$ and $c=1$ in order to transform it into a one-parameter family of equations that still unifies \eqref{CH} and \eqref{Nov}, and also admits the peaked wave solutions $u(t,x) = c^{1/k}e^{-|x-ct|},$ called peakon solutions \cite{CH}, where $c$ denotes the wave speed. Despite admitting an infinite number of conservation laws only when the equation is reduced to \eqref{CH} or \eqref{Nov}, it was not long before the interesting properties of \eqref{4par} attracted attention from researchers. In terms of applied analysis, Himonas and Holliman, in the same paper \cite{HH} showed that for any positive integer $k\geq 1$, $b=a+1$ and $c=1$, the equation is Hadamard well-posed in $H^s(\R)$ for $s>3/2$ and, more recently, Barostichi, Himonas and Petronilho \cite{BHP1} considered the choices $a=k+2,b=k+1$ and $c=1$ in \eqref{4par} to extend local well-posedness to global for the resulting equation and also understand the behaviours of global analytic solutions provided that the McKean quantity $m_0=m_0(x):=(1-\partial_x^2)u(0,x)$ does not change sign. For a geometric interpretation of the sign persistance of the McKean quantity and its consequences, see \cite{const} and discussions in \cite{CK,Kolev}.

It is important to observe that the restriction $k\geq 1$ in \eqref{4par} is due to two main reasons: firstly, the Camassa-Holm and Novikov equations are accomplished when we have two particular positive choices of $k$ and, secondly, problems with singularity obviously arise whenever considering $k<1$. Moreover, the former also explains why the constants $a,b,c$ are taken as different than zero. However, by allowing $b=0$ and $a=c=1$ in \eqref{4par}, one arrives at the equation
\begin{align}\label{gHS}
    m_t + u^km_x= 0,
\end{align}
where $k$ will be taken as a positive integer and $m=u-u_{xx}$. In the particular case where $k=1$, \eqref{gHS} is a very particular case of the $b$-equation $m_t + bmu_x + um_x=0$ considered in \cite{DHH} and later shown in \cite{HS1,HS} to have hydrodynamic applications when $b\neq -1$. Moreover, it can also be obtained from Kodama transformation to describe shallow water elevation \cite{holm-physica}.  
In terms of well-posedness, we observe that in \cite{HH,Yan} the authors showed that \eqref{gHS} is well-posed for an initial value $u_0\in H^{s}(\R),$ where $s>3/2$.

It is crucial to observe, however, that although local well-posedness of \eqref{gHS} in Sobolev \cite{HH}, Besov \cite{Yan} and Gevrey \cite{BHPlocal} spaces has been successfully established, not much else has been considered for $k>1$. In fact, the reasons for this fact are rather simple: the case $k=1$ in \eqref{gHS} is only known to conserve the momentum $\int_{\R} m(t,x)dx$ for rapidly decreasing solutions, which is equivalent to saying that
\begin{align}\label{1.0.5}
    \mathcal{H}(u) = \int_{\R}u(t,x)dx
\end{align}
is independent of time for the same sort of solution. For the generalized equation  \eqref{gHS} with $k>1$, the situation becomes even more drastic as no conservation laws seem to exist \textcolor{red}{\cite{Eu}}, which poses a difficulty that perhaps may be impossible to overcome in the attempt to study solutions and their properties.


One of the pioneering works is \cite{Eu1}, where in the particular case of $k=1$ in \eqref{gHS} the authors considered global well-posedness and deduced that, also making the assumption that the McKean quantity does not change sign, in $H^3(\R)$, it is possible to extend the local solutions and then the maximal time of existence is infinite. This is indeed a remarkable result once the equation \eqref{gHS} lacks the conservation of the $H^1(\R)$ norm and the construction of a highly non-trivial functional was required to show that the solution could not blow up at a finite time. Following a similar direction, in \cite{IgorNew} the authors determined global well-posedness for the periodic case and also studied continuation of periodic solutions. For the general case $k>1$ the authors in \cite{Eu1} also answered some of the questions raised by Himonas and Thompson \cite{HT}, giving a characterization of assymptotic behavior or solutions based on the initial data, and a blow-up criteria has been established in \cite{Yan}. The determination of global well-posedness for $k>1$, however, is still an open problem.

In this paper, we are interested in the initial value problem
\begin{align}\label{nonlocal}
u_t =F(u),\quad\quad\quad u(0,x) = u_0(x),
\end{align}
where
\begin{align}\label{1.0.7}
    F(u)=- \partial_x\left[\frac{u^{k+1}}{k+1} + \frac{3}{2}(1-\partial_x^2)^{-1}(ku^{k-1}u_x^2)\right] + (1-\partial_x^2)^{-1}\left[\frac{k(k-1)}{2}u^{k-2}u_x^3\right],
\end{align}
and complementing the results found in \cite{Eu1,IgorNew}. We observe that \eqref{nonlocal}--\eqref{1.0.7} is nothing but the evolution formulation of the Cauchy problem of \eqref{gHS} after the inversion of the Helmholtz operator $1-\partial_x^2$.

Consider the $L^2_x(\R)$ space of square integrable functions endowed with the norm
$$\Vert f\Vert_{L^2_x} = \left(\int_{\R}|f(x)|^2dx\right)^{1/2}.$$
The main function space of our interest in the present paper is the Gevrey space $G^{\sigma,s}(\R)$, where $\sigma>0$ and $s\in\R$, of functions in $L^2(\R)$ such that the norm
$$\Vert f\Vert_{G^{\sigma,s}} := \Vert (1+|\xi|^2)^{s/2}e^{\sigma|\xi|}\hat{f}(\xi)\Vert_{L^2_{\xi}} = \left(\int_{\R}(1+|\xi|^2)^s e^{2\sigma|\xi|}|\hat{f}(\xi)|^2d\xi\right)^{1/2}$$
is finite, where $\hat{f}$ denotes the Fourier transform
$$\hat{f}(\xi) = \frac{1}{\sqrt{2\pi}}\int_{\R}e^{-ix\xi}f(x)dx.$$
In the particular case where $\sigma\to 0$, the space $G^{0,s}(\R)$ becomes the usual Sobolev space $H^s(\R)$. In a result known as Paley-Wiener theorem (see \cite{Kat}), the Gevrey space $G^{\sigma,s}(\R)$ is characterized as the restriction to the real line of functions that are analytic on a strip of width $2\sigma$.


Our main intention is to show that well-posedness of \eqref{nonlocal}--\eqref{1.0.7} goes beyond Sobolev spaces in the sense of a  proof for global well-posedness in Gevrey spaces by making use of the Kato-Masuda \cite{KM} machinery and certain embeddings between spaces.

Before proceeding with our main result, we state a generalization of global well-posedness in Sobolev spaces for equation \eqref{nonlocal}--\eqref{1.0.7} with $k=1$. We observe that the result proven in \cite{Eu1} covers an initial data $u_0\in H^3(\R)$, while here we establish the result to $u_0\in H^s(\R)$, $s\geq 3$, which is enough for our purposes.

\begin{proposition}\label{teo.global}
Given $u_0\in H^s(\R)$, $s\geq 3$, if $m_0 \in L^1(\R)$ does not change sign, then the unique local solution $u$ for \eqref{nonlocal}--\eqref{1.0.7} with $k=1$ exists globally in $C([0,\infty);H^s(\R)) \cap C^1([0,\infty);H^{s-1}(\R))$.
\end{proposition}

With Proposition \ref{teo.global} in hand, we enunciate our main result.

\begin{theorem}\label{teo1.4}
    Given $u_0\in G^{1,s}(\R)$, with $s>5/2$, if $m_0(x)$ does not change sign, then the Cauchy problem of \eqref{nonlocal}--\eqref{1.0.7} with $k=1$ has a unique global analytic solution $u\in C^{\omega}([0,\infty)\times \R)$. 
\end{theorem}

In the context of hydrodynamic applications, analyticity is a crucial ingredient to prove an intrinsic characterization of symmetric waves, see \cite{Escher} for more details. Here, we observe that our unique space-time analytic solutions provided by Theorem \ref{teo1.4} are not necessarily traveling waves, which then provides an interesting and general result.


The proof of Theorem \ref{teo1.4} relies on the powerful machinery of Kato and Masuda \cite{KM} and embedding properties of certain spaces, see \cite{BHP1,KM}. Another useful space is an adaptation of the Banach spaces proposed by Himonas and Misiolek in \cite{misiolek}: for $\sigma>0$ and $m$ is a positive integer, the set  $E_{\sigma,m}(\R)$ of infinitely differentiable functions such that
\begin{align*}
    \vert\vert\vert f\vert\vert\vert_{E_{\sigma,m}} = \sup\limits_{j\in\mathbb{Z}_+}\frac{\sigma^j(j+1)^2}{j!}\Vert \partial_x^jf\Vert_{H^{2m}}<\infty
\end{align*}
is a Banach space by its own turn. In order to prove that the lifespan is infinite and the solution is analytic in both variables $t$ and $x$, we will make use of the auxiliary local well-posedness result. 

\begin{proposition}\label{teo1}
    Given $u_0(x)\in E_{\sigma_0,m}(\R)$, with $m\geq 3$ and for some $\sigma_0\in(0,1]$ fixed, for
    \begin{align}\label{lifespan}
      T = \kappa_m \frac{1}{\hm{u_0}{\sigma_0}{m}^k},
    \end{align}
    where
    $$\kappa_m = \frac{1}{\left[\frac{1}{k+1}+\frac{3k}{2} + \frac{k(k-1)}{2}\right](2^{2(k+2)}+8)c_m^k},$$ with $c_m>0$ depending only on $m$, and for every $\sigma \in (0,\sigma_0)$, the Cauchy problem for \eqref{nonlocal}--\eqref{1.0.7} has a unique solution $u$ that is analytic in the disc $D(0,T(\sigma_0-\sigma))$ with values in $E_{\sigma,m}(\R)$. Moreover, the bound
    $$\sup\limits_{|t|<T(1-\sigma)}\hm{u(t) - u_0}{\sigma}{m}< \hm{u_0}{\sigma_0}{m}$$
    holds.
\end{proposition}


The paper is organized as follows. In Section \ref{sec2} we establish the basic function spaces and auxiliary propositions required for the understanding and proofs of our results. In Section \ref{sec3}, we present the proof of Proposition \ref{teo1} for any $k\in \mathbb{Z}_+$, which follows from the technical estimates of Section \ref{sec2}. After that, in Section 4 we present a proof of Proposition \ref{teo.global}, which can also be found in \cite{Eu1} for $s=3$. Finally, in Section \ref{sec5} we provide a proof for global well-posedness in $H^{\infty}(\R)$ and finalize with the extensive and complex proof of Theorem \ref{teo1.4} by making use of Proposition \ref{teo1}.

\section{Function spaces and auxiliary results}\label{sec2}



In this section we will enunciate the theory behind the function spaces presented in the introduction. 

We start recalling that, similarly to Sobolev spaces, one interesting property of Gevrey spaces is that it is possible to continuously embed them based on the parameters $\sigma$ and $s$, see \cite{BHP}:
\begin{enumerate}
    \item If $0<\sigma'<\sigma$ and $s\geq 0$, then $\Vert\cdot \Vert_{G^{\sigma',s}}\leq \Vert\cdot \Vert_{G^{\sigma,s}}$ and $G^{\sigma,s}(\R)\hookrightarrow G^{\sigma',s}(\R)$;
    \item If $0<s'<s$ and $\sigma>0$, then $\Vert\cdot \Vert_{G^{\sigma,s'}}\leq \Vert\cdot \Vert_{G^{\sigma,s}}$ and $G^{\sigma,s}(\R)\hookrightarrow G^{\sigma,s'}(\R)$.
\end{enumerate}

Although our main result involves the use of Gevrey spaces, we will need to consider some auxiliary spaces and their embeddings. Following the work of Kato and Masuda \cite{KM} about the Korteweg-de Vries equation, for $r>0$ fixed we define the spaces $A(r)$ of functions that can be analytically extended to a function on a strip of width $r$, endowed with the norm
\begin{align}\label{2.0.1a}
    \Vert f\Vert_{\sigma,s}^2 = \sum\limits_{j=0}^{\infty} \frac{1}{j!^{2}}e^{2\sigma j} \Vert\partial_x^jf\Vert_{H^s},
\end{align}
for $s\geq 0$ and every $\sigma\in\R$ such that $e^{\sigma}<r$. Observe that if $r\geq r'$ then $f\in A(r)$ implies that $f\in A(r')$ and, therefore, $A(r)\subset A(r')$. 

For $H^{\infty}(\R) := \bigcap\limits_{s\geq 0}H^{s}(\R)$, we have the following sequence of embeddings (see Lemma 2.3 and Lemma 2.5 in \cite{BHP} and Lemma 2.2 in \cite{KM}):
\begin{align}\label{2.0.2}
    G^{\sigma,s}(\R)\hookrightarrow A(\sigma)\hookrightarrow H^{\infty}(\R),
\end{align}
for $\sigma>0$ and $s\geq 0$.


Similarly to Gevrey spaces, we have $E_{\sigma,m}(\R)\hookrightarrow E_{\sigma',m}(\R)$ for $0<\sigma'<\sigma$ and, more importantly, $E_{\sigma,m}(\R)$ is also continuously embedded into $H^{\infty}(\R)$ for all $m\geq 1$ and $\sigma>0$, see page 750 of \cite{BHP1}. Moreover, it is important to emphasize that if $1\leq m\leq m'$, then $\vert\vert\vert f\vert\vert\vert_{E_{\sigma,m}}\leq \vert\vert\vert f\vert\vert\vert_{E_{\sigma,m'}}$. 



To be able to extend regularity of global solutions, we will need to first consider local well-posedness in $E_{\sigma_0, m}(\R)$ for some $\sigma_0\in(0,1]$ and $m\geq 3$, and for that purpose some estimates will be required. The first one we enunciate is the algebra property, which allows us to relate the norm of multiplication to the multiplication of norms. 

\begin{lemma}[\textsc{Algebra property}]\label{lemma2.1}
    For any positive integer $m$, $0<\sigma\leq 1$ and $\varphi,\psi\in E_{\sigma,m}(\R)$, there is a positive constant $c_m$ depending only on $m$ such that
    $$\hm{\varphi\psi}{\sigma}{m}\leq c_s\hm{\varphi}{\sigma}{m}\hm{\psi}{\sigma}{m}.$$
\end{lemma}
\begin{proof}
The proof follows closely the lines of \cite{misiolek}.
\end{proof}

Consider an equation of the form $m_t = F(u,u_x,u_{xx},u_{xxx})$ and let $g(x) = e^{-|x|}/2$ \textcolor{red}{be} the Green function of the equation $(1-\partial_x^2)u = \delta(x)$, where $\delta$ denotes de Dirac delta distribution. Then we can write the inverse of the Helmholtz operator $1-\partial_x^2$ as
$$(1-\partial_x^2)^{-1}f(x) = g\ast f(x) = \frac{1}{2}\int_{\R}e^{-|x-y|}f(y)dy.$$
With respect to the spaces $E_{\sigma,m}(\R)$, the Helmholtz operator and its inverse have some important and useful properties that will be necessary to prove local well-posedness.

\begin{lemma}\label{lemma2.2}
    For $0<\sigma'<\sigma\leq 1$, $m\geq1$ and $\varphi\in E_{\sigma,m}(\R)$, then
    \begin{align}\label{eqlemma2.2}
        &\hm{\partial_x\varphi}{\sigma'}{m} \leq \frac{1}{\sigma-\sigma'}\hm{\varphi}{\sigma}{m},\\
        &\hm{\partial_x\varphi}{\sigma}{m} \leq \hm{\varphi}{\sigma}{m+1},\\
        &\hm{(1-\partial_x^2)^{-1}\varphi}{\sigma}{m+2}=\hm{\varphi}{\sigma}{m}.
    \end{align}
\end{lemma}
\begin{proof}
    The proofs of (2.4) and (2.5) follow immediately from the analogous estimates for Sobolev spaces and will be omitted, while the the proof of \eqref{eqlemma2.2}  requires an immediate adaptation of the proof of Lemma 2.4 (page 580) of \cite{misiolek}.
\end{proof}

In what follows, a function $u$ belongs to the space $C^{\omega}(I;X)$ if it is analytic in the interval $I$ as a function of $t$ and $u(t,\cdot)$ belongs to $X$. We will be interested in $C^{\omega}(I;G^{\sigma,s}(\R)), C^{\omega}(I;E_{\sigma,m}(\R))$ and $C^{\omega}(I;A(r))$. In the case where $u\in C^{\omega}(I\times \R)$, then $u(t,x)$ is analytic for $(t,x)\in I\times \R$.

The final result of this section, called Autonomous Ovsyannikov Theorem, will be used in the next section to prove Proposition \ref{teo1.4}. Its proof uses a very classical fixed point argument and follows closely the ideas in \cite{BHP,chines}, see also \cite{bao}.

\begin{proposition}[\textsc{Autonomous Ovsyannikov Theorem}]\label{teo1.5}
Let $X_{\delta}$ be a scale of decreasing Banach spaces for $0<\delta \leq 1$, that is, $X_\delta \subset X_{\delta'}, \Vert \cdot \Vert_{\delta'}\leq \Vert \cdot \Vert_{\delta}, 0<\delta'<\delta\leq 1,$ and consider the Cauchy problem
\begin{align}\label{1.0.8}
    \begin{cases}
    \displaystyle{\frac{du}{dt} = G(u(t))},\\
    u(0) = u_0.
    \end{cases}
\end{align}
Given $\delta_0\in(0,1]$ and $u_0\in X_{\delta_0}$, assume that $G$ satisfies the following conditions:
\begin{enumerate}
    \item For $0<\delta'<\delta<\delta_0$, $R>0$ and $a>0$, if the function $t\mapsto u(t)$ is holomorphic on $\{t\in \mathbb{C}; 0<|t|<a(\delta_0-\delta)$ with values in $X_{\delta}$ and $\sup\limits_{t<a(\delta_0-\delta)}\Vert u-u_0\Vert_{\delta}<R$, then the function $t\mapsto G(t,u(t))$ is holomorphic on the same set with values in $X_{\delta'}$.
    
    \item $G:X_{\delta}\rightarrow X_{\delta'}$ is well defined for any $0<\delta'<\delta<\delta_0$ and for any $R>0$ and  $u,v\in B(u_0,R)\subset X_{\delta}$, there exist positive constants $L$ and $M$ depending only on $u_0$ and $R$ such that
    \begin{align*}
        &\Vert G(u) - G(v)\Vert_{\delta'}\leq \frac{L}{\delta-\delta'}\Vert u-v\Vert_{\delta},\quad \Vert G(u_0)\Vert_{\delta}\leq \frac{M}{\delta_0-\delta},
    \end{align*}
\end{enumerate}
$0<\delta<\delta_0$. Then for
\begin{align}\label{lifespanAOT}
    T = \frac{R}{16LR+8M}
\end{align}
the initial value problem \eqref{1.0.8} has a unique solution $u\in C^{\omega}([0,T(\delta_0-\delta)),X_{\delta})$, for every $\delta\in (0,\delta_0)$, satisfying
\begin{align}\label{sup}
    \sup\limits_{|t|<T(\delta_0-\delta)}\Vert u(t)-u_0\Vert_{\delta}<R,\quad 0<\delta<\delta_0.
\end{align}
\end{proposition}

\section{Local well-posedness in the Himonas-Misiolek space}\label{sec3}

In this section we want to prove Proposition \ref{teo1} by making use of Autonomous Ovsyannikov Theorem. Before doing so in the next subsections, observe that the embeddings $E_{\sigma,m}(\R)\hookrightarrow E_{\sigma',m}(\R)$, for $0<\sigma'<\sigma$, guarantee that the function $$H(u) = - \partial_x\left[\frac{u^{k+1}}{k+1} + \frac{3}{2}(1-\partial_x^2)^{-1}(ku^{k-1}u_x^2)\right] + (1-\partial_x^2)^{-1}\left[\frac{k(k-1)}{2}u^{k-2}u_x^3\right],$$ taken as the right-hand side of \eqref{nonlocal}, is a well-defined function from $E_{\sigma,m}(\R)$ to $E_{\sigma',m}(\R)$ for every choice of $k$. Moreover, Condition 1 for the Autonomous Ovsyannikov Theorem is trivially satisfied. Therefore, it remains to prove Condition 2.

The proof of Proposition \ref{teo1} will be given in two parts. First we will separately prove the case where $k=1$, and then proceed to the case $k>1$. We would like to point out that Proposition \ref{teo1} holds for any positive choice of the parameter $k$, which then recovers the case $b=0$ for the $b$-equation.

\subsection{Proof for $k=1$}

Consider the function $F(u)$ given by \eqref{1.0.7} with $k=1$.
\begin{proposition}\label{prop1}
    Given $\sigma_0\in(0,1]$, $u_0\in E_{\sigma_0,m}(\R)$, with $m\geq3$, and $\sigma \in (0,\sigma_0)$, there exists a positive constant $M$ that depends only on $m$ and $u_0$ such that
    $$\hm{F(u_0)}{\sigma}{m}\leq \frac{M}{\sigma_0-\sigma}.$$
\end{proposition}
\begin{proof}
For $u_0\in E_{\sigma_0,m}(\R)$, write $$F(u_0) = - \frac{1}{2}\partial_x\left[u_0^{2}+ 3(1-\partial_x^2)^{-1}(\partial_xu_0)^2\right].$$

Using the triangle inequality, Lemma \ref{lemma2.2} and Lemma \ref{lemma2.1}, we obtain
\begin{align*}
\hm{F(u_0)}{\sigma}{m}\leq& \frac{1}{2}\hm{\partial_x u_0^2}{\sigma}{m} + \frac{3}{2}\hm{(1-\partial_x^2)^{-1}\partial_x(\partial_xu_0)^2}{\sigma}{m}\\
\leq& \frac{1}{2}\frac{c_m}{\sigma_0-\sigma}\hm{u_0}{\sigma_0}{m}^2 + \frac{3}{2}\frac{c_m}{\sigma_0-\sigma}\hm{\partial_x u_0}{\sigma_0}{m-2}^2\\
\leq&\frac{2c_m}{\sigma_0-\sigma}\hm{u_0}{\sigma_0}{m}^2 = \frac{M}{\sigma_0-\sigma},
\end{align*}
where $M=2c_m \hm{u_0}{\sigma_0}{m}^2$, finishing the proof.
\end{proof}

\begin{proposition}\label{prop2}
    Let $R>0$ and $\sigma_0\in(0,1]$. Given $u_0\in E_{\sigma_0,m}(\R)$, with $m\geq 3$, and $0<\sigma'<\sigma<\sigma_0$, if $u,v\in E_{\sigma,m}(\R)$ are such that
    $$\hm{u-u_0}{\sigma}{m}<R,\quad \hm{v-u_0}{\sigma}{m}<R,$$
    then there exists a positive constant $L$ that depends only on $m$, $u_0$ and $R$ such that
    $$\hm{F(u)-F(v)}{\sigma'}{m}\leq \frac{L}{\sigma-\sigma'}\hm{u-v}{\sigma}{m}.$$
\end{proposition}
\begin{proof}
From the triangle inequality, we have
$$\hm{F(u)-F(v)}{\sigma'}{m}\leq \frac{1}{2}\hm{\partial_x(u^2-v^2)}{\sigma'}{m} + \frac{3}{2}\hm{\partial_x(1-\partial_x^2)^{-1}(u_x^2-v_x^2)}{\sigma'}{m}.$$
By observing that $$u^2 - v^2 = (u-v)(u+v),\quad u_x^2 - v_x^2 = [\partial_x (u-v)][\partial_x (u+v)],$$
Lemma \ref{lemma2.2} and Lemma \ref{lemma2.1} yield
\begin{align*}
    \hm{F(u)-F(v)}{\sigma'}{m}\leq& \frac{1}{2}\frac{c_m}{\sigma-\sigma'}\hm{u-v}{\sigma}{m}\hm{u+v}{\sigma}{m}\\
    &+ \frac{3}{2} \frac{c_m}{\sigma-\sigma'}\hm{\partial_x(u-v)}{\sigma}{m-2}\hm{\partial_x(u+v)}{\sigma}{m-2}\\
    \leq& 2\frac{c_m}{\sigma-\sigma'}\hm{u-v}{\sigma}{m}\hm{u+v}{\sigma}{m},
\end{align*}
where in the last inequality we used the fact that $\hm{\partial_x(u-v)}{\sigma}{m-2}\leq \hm{u-v}{\sigma}{m}$ and an analogous estimate for $\partial_x(u+v)$.

Since $$\hm{u+v}{\sigma}{m}\leq \hm{u-u_0}{\sigma}{m} + \hm{v-u_0}{\sigma}{m} + 2 \hm{u_0}{\sigma}{m} <2(R +\hm{u_0}{\sigma_0}{m}),$$
we conclude that for $L = 4c_m(R + \hm{u_0}{\sigma_0}{m})$ the bound
\begin{align*}
    \hm{F(u)-F(v)}{\sigma'}{m} \leq \frac{L}{\sigma-\sigma'}\hm{u-v}{\sigma}{m}
\end{align*}
holds for $m\geq 2$ and $0<\sigma'<\sigma<\sigma_0$, completing the proof.
\end{proof}

We are now in conditions to finalize the proof of Proposition \ref{teo1} for $k=1$. For this purpose, observe that in Proposition \ref{prop1} we have $M = 2c_m \hm{u_0}{\sigma_0}{m}^2,$ while in Proposition \ref{prop2} we have $L = 4c_m (R+\hm{u_0}{\sigma_0}{m})$ for any $R>0$. Letting $C = 4c_m,$ then we can write $L = C(R+\hm{u_0}{\sigma_0}{m})$ and $M = \frac{C}{2}\hm{u_0}{\sigma_0}{m}^2.$

From propositions \ref{prop1} and \ref{prop2}, the conditions for the autonomous Ovsyannikov theorem are satisfied and, therefore, for $m\geq 3$ and  $T$ given by \eqref{lifespanAOT} 
there exists a unique solution $u$ to the Cauchy problem \eqref{nonlocal} which for every $\sigma \in (0,\sigma_0)$ is a holomorphic function in $D(0,T(\sigma_0-\sigma))$ to $E_{\sigma,m}(\R)$ and satisfies \eqref{sup}.
Taking $R = \hm{u_0}{\sigma_0}{m}$ yields
$$T = \frac{1}{144 c_m}\times \frac{1}{\hm{u_0}{\sigma_0}{m}}$$
and the proof of existence and uniqueness of Proposition \ref{teo1} is finished for $k=1$.

\subsection{Proof for $k> 1$}\label{sec3.2}

For the case $k>1$, we will make use of the simple algebraic inequality
\begin{align}\label{4.0.1}
    3+2^{k-3}<2^k.
\end{align}

Consider the function $F(u)$ given by \eqref{1.0.7}. Similarly to Proposition \ref{prop1} and Proposition \ref{prop2} for the case $k=1$, we will estimate $\hm{F(u_0)}{\sigma}{m}$ and $\hm{F(u)-F(v)}{\sigma'}{m}$ for $m\geq 3$ and $0<\sigma'<\sigma<\sigma_0\leq 1$.

\begin{proposition}\label{prop3}
     Given $\sigma_0\in(0,1]$, $u_0\in E_{\sigma_0,m}(\R)$, with $m\geq3$, and $\sigma \in (0,\sigma_0)$, there exists a positive constant $M$ that depends only on $m$ and $u_0$ such that
    $$\hm{F(u_0)}{\sigma}{m}\leq \frac{M}{\sigma_0-\sigma}.$$
\end{proposition}
\begin{proof}
    Given $u_0\in E_{\sigma_0,m}(\R)$, using \eqref{1.0.7} we have
    \begin{align*}
    \hm{F(u_0)}{\sigma}{m} \leq&\frac{1}{k+1}\hm{\partial_xu_0^{k+1}}{\sigma}{m} +\frac{3k}{2}\hm{\partial_x(1-\partial_x^2)^{-1}u_0^{k-1}(\partial_xu_0)^2}{\sigma}{m}\\
    &+\frac{k(k-1)}{2}\hm{(1-\partial_x^2)^{-1}u_0^{k-2}(\partial_xu_0)^3}{\sigma}{m}.
    \end{align*}
    From Lemma \ref{lemma2.2} and the algebra property, we can write
    \begin{align*}
        \hm{\partial_x u_0^{k+1}}{\sigma}{m} \leq& \frac{1}{\sigma_0-\sigma}\hm{u_0^{k+1}}{\sigma_0}{m}\leq \frac{c_m^k}{\sigma_0-\sigma}\hm{u_0}{\sigma_0}{m}^{k+1},\\
        \hm{\partial_x(1-\partial_x^2)^{-1}u_0^{k-1}(\partial_xu_0)^2}{\sigma}{m}\leq& \frac{1}{\sigma_0-\sigma}\hm{u_0^{k-1}(\partial_xu_0)^2}{\sigma_0}{m-2}\leq \frac{c_m^k}{\sigma_0-\sigma}\hm{u_0}{\sigma_0}{m}^{k+1},\\
        \hm{(1-\partial_x^2)^{-1}u_0^{k-2}(\partial_xu_0)^3}{\sigma}{m}\leq& \hm{u_0^{k-2}(\partial_xu_0)^3}{\sigma}{m-2}\leq \frac{c_m^k}{\sigma_0-\sigma}\hm{u_0}{\sigma_0}{m}^{k+1}.
    \end{align*}
    Thus,
    \begin{align*}
        \hm{F(u_0)}{\sigma}{m}\leq& \left[\frac{1}{k+1}+\frac{3k}{2}+\frac{k(k-1)}{2}\right]\frac{c_m^k}{\sigma_0-\sigma}\hm{u_0}{\sigma_0}{m}^{k+1}.
    \end{align*}
    By letting $$M = \left[\frac{1}{k+1}+\frac{3k}{2}+\frac{k(k-1)}{2}\right]c_m^k\hm{u_0}{\sigma_0}{m}^{k+1},$$
    we finally conclude that $$\hm{F(u_0)}{\sigma}{m}\leq \frac{M}{\sigma_0-\sigma},$$
    for $0<\sigma<\sigma_0$, and the result is proven.
\end{proof}

Before proceeding with the next estimate, it is necessary to state a result that only requires the triangle inequality and successive applications of the algebra property.
\begin{lemma}\label{lemma4.1}
    For $u,v\in E_{\sigma,m}(\R)$, with $\sigma>0$ and $m\geq 1$, let $$f_k(u,v) = \sum\limits_{j=0}^ku^jv^{k-j}.$$
    Then there exists a positive constant $c_m$ depending only on $m$ such that $$\hm{f_k(u,v)}{\sigma}{m}\leq c_m^{k-1}(\hm{u}{\sigma}{m} + \hm{v}{\sigma}{m})^k.$$
\end{lemma}

We shall now proceed with the last crucial estimate required to make use of the Autonomous Ovsyannikov Theorem and finish the proof of Proposition \ref{teo1}.

\begin{proposition}\label{prop4}
    Let $R>0$ and $\sigma_0\in(0,1]$. Given $u_0\in E_{\sigma_0,m}(\R)$, with $m\geq 3$ and $0<\sigma'<\sigma<\sigma_0$, if $u,v\in E_{\sigma,m}(\R)$ are such that
    $$\hm{u-u_0}{\sigma}{m}<R,\quad \hm{v-u_0}{\sigma}{m}<R,$$
    then there exists a positive constant $L$ that depends only on $m$, $u_0$ and $R$ such that
    $$\hm{F(u)-F(v)}{\sigma'}{m}\leq \frac{L}{\sigma-\sigma'}\hm{u-v}{\sigma}{m}.$$
\end{proposition}
\begin{proof}
    Given $R>0$, $\sigma_0\in(0,1]$ and $u_0\in E_{\sigma_0,m}(\R)$, with $m\geq 3$, let $0<\sigma'<\sigma<\sigma_0$. In terms \eqref{1.0.7}, we write
    \begin{align}\label{4.0.2}
    \begin{aligned}
    \hm{F(u)-F(v)}{\sigma'}{m} \leq& \frac{1}{k+1}\hm{\partial_x(u^{k+1}-v^{k+1})}{\sigma'}{m}\\
    &+ \frac{3k}{2}\hm{\partial_x(1-\partial_x^2)^{-1}(u^{k-1}u_x^2-v^{ k-1}v_x^2)}{\sigma'}{m}\\
    &+\frac{k(k-1)}{2}\hm{(1-\partial_x^2)^{-1}(u^{k-2}u_x^3-v^{k-2}v_x^3)}{\sigma'}{m}.
    \end{aligned}
    \end{align}
    Since $u^{k+1}-v^{k+1} = (u-v)f_k(u,v)$, from Lemma \ref{lemma2.2} and Lemma \ref{lemma4.1} we obtain
    \begin{align*}
        \hm{\partial_x(u^{k+1}-v^{k+1})}{\sigma'}{m}\leq& \frac{1}{\sigma-\sigma'}\hm{(u-v)f_k(u,v)}{\sigma}{m}\\
        \leq& \frac{c_m^k}{\sigma-\sigma'}\left(\hm{u}{\sigma}{m}+\hm{v}{\sigma}{m} \right)^k\hm{u-v}{\sigma}{m}.
    \end{align*}
    For the second term, write
    \begin{align*}
    u^{k-1}u_x^2 - v^{k-1}v_x^2 =u^{k-1}[u_x-v_x][u_x+v_x] + (u-v)v_x^2f_{k-2},
    \end{align*}
    which, together with the triangle inequality, the algebra property and Proposition \ref{prop2}, yield
    \begin{align*}
        \hm{\partial_x&(1-\partial_x^2)^{-1}(u^{k-1}u_x^2-v^{k-1}v_x^2)}{\sigma'}{m}\\
        \leq& \frac{1}{\sigma-\sigma'}\left(\hm{u^{k-1}(u_x-v_x)(u_x+v_x)}{\sigma}{m-2} +\hm{(u-v)v_x^2f_{k-2}(u,v)}{\sigma}{m-2}\right)\\
        \leq&\frac{c_m}{\sigma-\sigma'}\left(\hm{u^{k-1}u_x+v_x}{\sigma}{m-2}+\hm{v_x^2f_{k-2}(u,v)}{\sigma}{m-2}\right)\hm{u-v}{\sigma}{m}\\
        \leq& \frac{c_m^k}{\sigma-\sigma'}\left(\hm{u}{\sigma}{m}^{k-1}\hm{u+v}{\sigma}{m}+(\hm{u}{\sigma}{m}+\hm{v}{\sigma}{m})^{k-2}\hm{v}{\sigma}{m}^2\right)\\
        &\times\hm{u-v}{\sigma}{m}.
    \end{align*}
    From the proof of Proposition \ref{prop2} we know that $\hm{u+v}{\sigma}{m}<2(R+\hm{u_0}{\sigma_0}{m}).$ Moreover, we also have
    \begin{align}\label{4.0.3}
    \hm{u}{\sigma}{m}\leq \hm{u-u_0}{\sigma}{m}+\hm{u_0}{\sigma}{m}<R + \hm{u_0}{\sigma_0}{m},
    \end{align}
    which tells that
    \begin{align*}
        \hm{\partial_x(1-\partial_x^2)^{-1}&(u^{k-1}u_x^2-v^{k-1}v_x^2)}{\sigma'}{m}\\
        &\leq (2+2^{k-2})\frac{c_m^k}{\sigma-\sigma'} (R + \hm{u_0}{\sigma_0}{m})^k\hm{u-v}{\sigma}{m}
    \end{align*}
    and
    \begin{align*}
        \hm{\partial_x(u^{k+1}-v^{k+1})}{\sigma'}{m}\leq& 2^k\frac{c_m^k}{\sigma-\sigma'}\left(R + \hm{u_0}{\sigma_0}{m}\right)^k\hm{u-v}{\sigma}{m}.
    \end{align*}
    To deal with the third and last term on the right-hand side of \eqref{4.0.2}, observe that
    \begin{align*}
    u^{k-2}u_x^3 - v^{k-2}v_x^3 &= u^{k-2}(u_x^3-v_x^3) + (u^{k-2}-v^{k-2})v_x^3\\
    &=u^{k-2}[u_x-v_x][u_x^2+ u_x v_x + v_x^2]+(u-v)v_x^3f_{k-3}(u,v).
    \end{align*}
    Thus, we can write
    \begin{align*}
        \hm{(1-&\partial_x^2)^{-1}(u^{k-2}u_x^3 - v^{k-2}v_x^3)}{\sigma'}{m}\\
        \leq& \hm{u^{k-2}(u_x-v_x)(u_x^2+u_xv_x+v_x^2)}{\sigma'}{m-2}      +\hm{(u-v)v_x^3f_{k-3}(u,v)}{\sigma'}{m-2}\\
        \leq& c_m(\hm{u^{k-2}(u_x^2+u_xv_x+v_x^2)}{\sigma'}{m-2}\hm{\partial_x(u-v)}{\sigma'}{m-2}\\
        &+\hm{v_x^3f_{k-3}(u,v)}{\sigma'}{m-2}\hm{u-v}{\sigma'}{m-2})\\
        \leq& \frac{c_m^k}{\sigma-\sigma'}\left[\hm{u}{\sigma}{m}^k + \hm{u}{\sigma}{m}^{k-1}\hm{v}{\sigma}{m} + \hm{u}{\sigma}{m}^{k-2}\hm{v}{\sigma}{m}^2\right.\\
        &\left.+(\hm{u}{\sigma}{m}+\hm{v}{\sigma}{m})^{k-3}\hm{v}{\sigma}{m}^3\right]\hm{u-v}{\sigma}{m}.
    \end{align*}
    From the estimate \eqref{4.0.3} it is then obtained
    \begin{align*}
        \hm{(1-\partial_x^2)^{-1}&(u^{k-2}u_x^3 - v^{k-2}v_x^3)}{\sigma'}{m}\\
        \leq& (3+2^{k-3})\frac{c_m^k}{\sigma-\sigma'}(R + \hm{u_0}{\sigma_0}{m})^k\hm{u-v}{\sigma}{m}.
    \end{align*}
    Now under substitution of the respective terms in \eqref{4.0.2}, we arrive at
    \begin{align*}
        \hm{F(u)-F(v)}{\sigma}{m} \leq& \frac{c_m^k}{\sigma-\sigma'}\left[\frac{2^k}{k+1} + (2+2^{k-2})\frac{3k}{2}+(3+2^{k-3})\frac{k(k-1)}{2}\right]\\
        &\times(R+\hm{u_0}{\sigma_0}{m})^k\hm{u-v}{\sigma}{m}.
    \end{align*}
    Observe now that for $k>1$ we have $2^{k-2}<2^{k-1}$, $2\leq2^{k-1}$ and, from \eqref{4.0.1}, $3+2^{k-3}<2^k$. It means that the last inequality can be written as
    $$\hm{F(u)-F(v)}{\sigma'}{m}\leq \frac{L}{\sigma-\sigma'}\hm{u-v}{\sigma}{m},$$
    where L = $C(R+\hm{u_0}{\sigma_0}{m})^k,$ with $C = 2^k\left[\frac{1}{k+1}+\frac{3k}{2}+\frac{k(k-1)}{2}\right]c_m^k$, and the proof is finished.
\end{proof}

We will now proceed with the final part of the proof of Proposition \ref{teo1}. For $M = \frac{C}{2^k}\hm{u_0}{\sigma_0}{m}^{k+1}$ and $R=\hm{u_0}{\sigma_0}{m}$, from the Autonomous Ovsyannikov Theorem, for 
\begin{align*}
T =& \frac{R}{16LR+8M}=\frac{1}{\left[\frac{1}{k+1}+\frac{3k}{2}+\frac{k(k-1)}{2}\right](2^{2(k+2)}+8)c_m^k}\frac{1}{\hm{u_0}{\sigma_0}{m}^k},
\end{align*}
 there exists a unique solution $u$ to the Cauchy problem of \eqref{nonlocal} which, for every $\sigma\in(0,\sigma_0)$, is a holomorphic function in $D(0,T(\sigma_0-\sigma))$ into $E_{\sigma_0,m}(\R)$. Therefore, the proof of Proposition \ref{teo1} is complete for any positive integer $k$. Observe that taking $k=1$ will result in the same $T$ obtained last section, which shows that it indeed unifies both cases.

Since Lemma \ref{lemma2.1} and a similar Lemma \ref{lemma2.2} are still valid for Gevrey spaces $G^{\sigma,s}(\R)$, where $0<\sigma<\sigma'\leq \sigma_0\leq 1$ and $s>1/2$, see \cite{BHP,chines}, a repetition of the same calculations for $\sigma_0=1$ provides an analogous result for these spaces, which will be useful and is therefore stated in the next result. For an alternative proof, see Theorem 1 in \cite{BHPlocal}.

\begin{corollary}\label{cor3.1}
    Given $u_0(x):=u(0,x)\in G^{1,s}(\R)$, with $s\geq 5/2$, there exists $T>0$ such that for every $\sigma \in (0,1)$ the Cauchy problem for \eqref{nonlocal} has a unique solution $u\in C^{\omega}([0,T(1-\sigma)); G^{\sigma,s}(\R))$.
\end{corollary}

\section{Global well-posedness in Sobolev spaces}

In this section we will prove Proposition \ref{teo.global} and the proof will be based on the local well-posedness in Sobolev spaces, a certain estimate for the $H^3(\R)$ norm of the local solution and the Sobolev embedding theorem. It is worth mentioning that the proof for the case $s=3$ is already proven in \cite{Eu1} and, therefore, presented here just for the sake of completeness. Firstly we enuntiate a result due to Yan \cite{Yan}, see also Himonas and Holliman \cite{HH}.

\begin{lemma}\label{lemma5.1'}
    Suppose that $u_0\in H^s(\R), s>3/2$. There exist a maximal time of existence $T>0$ and a unique solution $u\in C([0,T);H^s(\R))\cap C^1([0,T);H^{s-1}(\R))$ of \eqref{nonlocal}--\eqref{1.0.7} with $k=1$. Moreover, the solution $u$ satisfies the following energy estimate:
    \begin{align}\label{estimate}
    \frac{d}{dt}\Vert u\Vert_{H^ s} \leq c_s \Vert u\Vert_{C^1}\Vert u\Vert_{H^s},
    \end{align}
    for some positive constant $c_s$. Finally, the data-to-solution map $u(0)\mapsto u(t)$ is continuous.
\end{lemma}
\begin{proof}
    For the proof of existence and uniqueness of solution, see Corollary 2.1 of \cite{Yan}, while the estimate \eqref{estimate} is given by (2.29) of \cite{HH}.
\end{proof}

After having the energy estimate \eqref{estimate} guaranteed, we enunciate a result stated as part of Lemma 5.1 and Theorem 3.1 of da Silva and Freire \cite{Eu1}.

\begin{lemma}\label{lemma5.2'}
    Given $u_0\in H^3(\R)$, let $u$ be the corresponding unique solution of \eqref{nonlocal}--\eqref{1.0.7} with $k=1$.
    \begin{enumerate}
        \item[(a)] If there exists $\kappa>0$ such that $u_x>-\kappa$, then $\Vert u\Vert_{H^3}\leq e^{\kappa t/2}\Vert u_0\Vert_{H^3}$;
        
        \item[(b)] If $m_0$ does not change sign, then $-u_x \leq \Vert m_0\Vert_{L^1}$ for each $(t,x)\in [0.T)\times \R$.
    \end{enumerate}
\end{lemma}

As a consequence, we  can extend the last Lemma to and initial data in $H^s(\R)$ for $s\geq 3$ as the following Corollary states:

\begin{corollary}\label{cor5.1'}
Let $u_0\in H^s(\R), s\geq 3,$ be an initial data with corresponding local solution $u$. If $m_0$ does not change sign, then
$$\Vert u\Vert_{H^3}\leq e^{\kappa t/2}\Vert u_0\Vert_{H^3},\quad \text{for some}\,\, 0<\kappa<\infty.$$
\end{corollary}
\begin{proof}
    Since $s\geq 3,$ we have $H^s(\R)\subset H^3(\R)$ and $H^{s-1}(\R)\subset H^2(\R)$. Therefore, $u_0\in H^3(\R)$ and, from Lemma \ref{lemma5.2'}(b), there exists $\kappa = \Vert m_0\Vert_{L^1}$ such that $-u_x < \kappa$. The result now follows from Lemma \ref{lemma5.2'}(a).
\end{proof}

We are now ready to prove Theorem \ref{teo.global}.

\textbf{Proof of Proposition \ref{teo.global}:} For $u_0\in H^s(\R), s\geq 3,$ let $u\in C([0,T);H^s(\R))\cap C^1([0,T);H^{s-1}(\R))$ be the unique local solution. From Lemma \ref{lemma5.1'}, the solution is such that \eqref{estimate} holds for $0\leq t< T$. From Gr\"onwall's inequality, we have
\begin{align}\label{eq5.2'}
\Vert u\Vert_{H^s}\leq \Vert u_0\Vert_{H^s} e^{c_s \int_0^t \Vert u(\tau)\Vert_{C^1}d\tau}.
\end{align}
Since $m_0$ does not change sign, from the Sobolev embedding theorem we have
\begin{align*}
    \Vert u\Vert_{C^1} = \Vert u\Vert_{L^{\infty}} + \Vert u_x\Vert_{L^{\infty}}\leq \Vert u\Vert_{H^s}+\Vert u\Vert_{H^{s+1}},
\end{align*}
for $s>1/2$. Taking $s=2$ and using Corollary \ref{cor5.1'}, we obtain
$$\Vert u\Vert_{C^1}\leq  2\Vert u \Vert_{H^3}\leq 2 e^{\kappa t/2}\Vert u_0\Vert_{H^3}.$$
Note that $u_0\in H^s(\R)$ for $s\geq 3$ tells that $u_0\in H^ 3(\R)$ and then, under substitution in \eqref{eq5.2'}, the condition becomes
\begin{align}
\Vert u\Vert_{H^s}&\leq \Vert u_0\Vert_{H^s} e^{c_s \Vert u_0\Vert_{H^3}\int_0^t e^{\kappa \tau/2}d\tau}\\
&= \Vert u_0\Vert_{H^s} e^{c_s \Vert u_0\Vert_{H^3}\left(\frac{e^{K t}-1}{K}\right)}\\
&< \Vert u_0\Vert_{H^s} e^{c_s \Vert u_0\Vert_{H^3}e^{K t}},
\end{align}
where $K = \kappa/2$. This means that $u$ does not blow-up at a finite time and the solution $u$ can be extended globally in time.

\section{Global well-posedness and radius of spatial analyticity}\label{sec5}

In this section we prove Theorem \ref{teo1.4}. In what follows, we will consider the initial value problem
\begin{align}\label{nonlocal1}
\begin{cases}
    u_t = -\frac{1}{2}\partial_x[u^2 + 3(1-\partial_x^2)^{-1}u_x^2]=:F(u), & t\geq 0,\quad x\in \R,\\
    u(0,x)=u_0(x),
\end{cases}
 \end{align}
and make use of local and global well-posedness in Sobolev spaces to extend regularity. The machinery here presented follows closely the ideas of Kato and Masuda \cite{KM} and later Barostichi, Himonas and Petronilho \cite{BHP1}. Since the proof of Theorem \ref{teo1.4} is extremely technical and extensive, we opt to divide the result in several propositions that together will give our desired result. The propositions that will be presented next will be proven in the next subsections. We start with a very important result regarding global well-posedness in $H^{\infty}(\R)$.

\begin{proposition}\label{prop5.1}
    If $u_0\in G^{1,s}(\R)$, $s>3/2$, and $m_0$ does not change sign, then \eqref{nonlocal1} has a unique solution $u\in C([0,\infty);H^{\infty}(\R))$.
\end{proposition}

Once we have the global solution established, we will extend regularity to the Kato-Masuda space. We are able to find $r_1>0$ such that for each fixed arbitrary time $T>0$ the solution will belong to $A(r_1(t))$ as a space function for $t\in[0,T]$. From the definition of the spaces $A(r)$, this $r_1$ will be the radius of spatial analyticity of the solution.

\begin{proposition}\label{teo5.1}
    Given $u_0\in G^{1,s}(\R)$, with $s>5/2$, suppose that $m_0$ does not change sign and let $u\in C([0,\infty);H^{\infty}(\R))$ be the unique solution to the initial value problem of \eqref{nonlocal1}. Then there exists $r_1>0$ such that $u\in C([0,\infty);A(r_1))$. Moreover, for every $T>0$ an explicit lower bound for the radius of spatial analyticity is given by
    $$r_1(t)\geq L_3e^{-L_1e^{L_2t}},\quad t\in[0,T],$$
    where $L_1=\frac{52\sqrt{2}}{7}\km{u_0}{\sigma_0}{2}$ for $\sigma_0<0$ fixed, $L_2 = 112\mu, L_3 = r(0)e^{L_1}$ and $\mu = 1+\max\{\sobolev{u}{2};t\in[0,T]\}$. 
\end{proposition}

Proposition \ref{teo5.1} says that the global solution is analytic in $x$ and gives a lower bound for the radius of spatial analyticity. 
The next step is to extend regularity to $t$. Our first step towards this goal is to prove that the solution $u$ is locally analytic in time, as enunciated by the next result.

\begin{proposition}\label{prop5.3}
    Given $u_0\in G^{1,s}(\R)$, with $s>5/2$, let $u\in C([0,\infty); A(r_1))$ be the unique solution of \eqref{nonlocal1}. Then there exist $T>0$ and $\delta(T)>0$ such that the unique solution $u$ belongs to $C^{\omega}([0,T];A(\delta(T)))$.
\end{proposition}

Once local analyticity is established, we show that the analytic lifespan is infinite.

\begin{proposition}\label{prop5.4}
    For the unique solution $u \in C^{\omega}([0,T];A(\delta(T)))$, we have
    $$T^{\ast} = \sup\{T>0, u\in C^{\omega}([0,T]; A(\delta(T))),\,\,\text{for some}\,\, \delta(T)>0\}=\infty.$$
\end{proposition}

Finally, to conclude our result, we use a result proved by Barostichi, Himonas and Petronilho in \cite{BHP1} (see page 752).

\begin{lemma}\label{lemma5.1}
    If $u\in C^{\omega}([0,T]; A(r(T)))$ for all $T>0$ and some $r(T)>0$, then $u \in C^{\omega}([0,\infty)\times \R)$.
\end{lemma}

\textbf{Proof of Theorem \ref{teo1.4}.} The proof is now reduced to a recollection of the previous propositions. Given $u_0\in G^{1,s}(\R)$, if $m_0$ does not change sign, from Proposition \ref{prop5.1} we have a unique solution $u\in C([0,\infty), H^{\infty}(\R))$. From Proposition \ref{teo5.1}, we guarantee the existence of $r_1>0$ such that $u \in C([0,\infty),A(r_1))$, which by Proposition \ref{prop5.3} belongs to $C^{\omega}([0,T],A(\delta(T)))$ for certain $T>0$ and $\delta(T)>0$. Proposition \ref{prop5.4} then guarantees that $u\in C^{\omega}([0,T],A(\delta(T)))$ for every $T>0$ and then Lemma \ref{lemma5.1} concludes that the solution $u$ is global analytic for both variables.

Moreover, we observe from Proposition \ref{teo5.1} that given $T>0$, we have $u(t)\in A(r_1)$ for $t\in [0,T]$ and $r_1(t) \geq L_3e^{-L_1e^{L_2t}}$. By means of the forthcoming expression \eqref{5.2.2} obtained in the proof of Proposition \ref{prop5.1} we can determine the radius of spatial analyticity as
$$r_1(t) = Ce^{-Ae^{Bt}},$$
where
$$A=\frac{26\sqrt{2}}{7\mu}(1+\mu)\km{u_0}{\sigma_0}{2},\quad B=112\mu,\quad C= e^{\sigma_0+A}$$
and $\mu, \sigma_0$ are given as in Proposition \ref{teo5.1}.

\subsection{Proof of Proposition \ref{prop5.1}}

For the proof of Proposition \ref{prop5.1}, the only ingredients required are Proposition \ref{teo.global}, the embeddings $G^{1,s}(\R)\subset H^{\infty}(\R)$ and $H^s(\R)\subset H^{s'}(\R)$ for $s>s'$, as shown next.

Since $u_0\in G^{1,s}(\R)$, from the embedding $G^{1,s}(\R)\subset H^{\infty}(\R)$ the initial data belongs, in particular, to $H^s(\R)$ for any $s\geq 3$. From Theorem \ref{teo.global}, there exists a unique global solution $u$ in $C([0,\infty);H^{s}(\R))\cap C^1([0,\infty);H^{s-1}(\R))$ for $s\geq 3$, which means that $u(t,\cdot) \in \bigcap\limits_{s\geq 3}H^s(\R)$. Now, the embedding $H^3(\R)\subset H^{s'}(\R)$ for $s'\in [0,3]$ shows that $u(t,\cdot)\in H^{s'}(\R)$ and $u(t,\cdot)\in H^{\infty}(\R)$, concluding the proof of Proposition \ref{prop5.1}

\subsection{Proof of Proposition \ref{teo5.1}}

This is by far the most technical and complicated result. 
The proof consists of bounding a certain inner product and using properties of dense spaces to find such $r_1$. For $m\geq 0$, it will be more convenient to consider an auxiliary norm 
$$\kma{u}{\sigma}^2 = \sum\limits_{j=0}^m\frac{1}{(j!)^{2}}e^{2\sigma j} \sobolev{\partial_x^ju}{2}^2$$
in $A(r)$ and recover the usual norm \eqref{2.0.1a} as we make $m\to \infty$. Moreover, we observe that $\kma{u}{\sigma}\leq \km{u}{\sigma}{2}$.

For our initial value problem \eqref{nonlocal1}, we note that, given $m\geq 0$, the function $F:H^{m+5}(\R)\rightarrow H^{m+2}(\R)$ is well-defined and continuous. Therefore, for $Z=H^{m+5}(\R)$ and $X= H^{m+2}(\R)$ the following result, which will be called Kato-Masuda Theorem, is valid, see Theorem 1 in \cite{KM} or Theorem 4.1 in \cite{BHP1} for more general formulations.

\begin{lemma}[\textsc{Kato-Masuda}]\label{prop4.2}
    Let $\{\Phi_{\sigma}:-\infty<\sigma<\bar{\sigma}\}$ be a family of real functions defined on an open set $O\subset Z$ for some $\bar{\sigma}\in\R$. Suppose that $F: O \rightarrow X$ is continuous, where $F$ is the function given by \eqref{nonlocal1} and
    \begin{enumerate}
        \item[(a)] $D\Phi_{\cdot}(\cdot):\R\times Z\rightarrow \mathcal{L}(\R\times X; \R)$ given by
        $$D\Phi_{\sigma}(v)F(v):=\langle F(v)\,,\, D \Phi_{\sigma}(v)\rangle$$
        is continuous, where $D$ denotes the Fréchet derivative;
        \item[(b)] there exists $\bar{r}>0$ such that $$D\Phi_{\sigma}(v)F(v) \leq \beta(\Phi_{\sigma}(v)) + \alpha(\Phi_{\sigma}(v))\partial_{\sigma}\Phi_{\sigma}(v),$$
        for all $v\in O$ and some nonnegative continuous real functions $\alpha(r)$ and $\beta(r)$ well-defined for $-\infty<r<\bar{r}$.
    \end{enumerate}
    For $T>0$, let $u\in C([0,T];O)\cap C^1([0,T];X)$ be a solution of the initial value problem \eqref{nonlocal1} such that there exists $b<\bar{\sigma}$ with $\Phi_{b}(u_0)<\bar{r}$. Finally, let $\rho(v)$ be the unique solution of
    \begin{align*}
        \begin{cases}
        \displaystyle{\frac{d\rho(t)}{dt} = \beta(\rho)},\\
        \rho(0) = \Phi_b(u_0),& t\geq 0.
        \end{cases}
    \end{align*}
    Then for
    $$\sigma(t) = b - \int_{0}^t \alpha(\rho(\tau))d\tau,\quad t\in[0,T_1],$$
    where $T_1>0$ is the lifespan of $\rho$, we have
    \begin{align}\label{5.0.2}
        \Phi_{\sigma(t)}(u) \leq \rho(t),\quad t\in[0,T]\cap[0,T_1].
    \end{align}
\end{lemma}

We observe the complexity of the Kato-Masuda Theorem and the amount of hypothesis required for the final result. It is important to mention as well that the procedure to prove our desired Proposition \ref{teo5.1} goes through Kato-Masuda Theorem and \eqref{5.0.2}, see also Proposition 4.1 of \cite{BHP1}.

However, one of the main issues is to establish the bound of item $(b)$. Before doing so, we shall define convenient parameters and functions that will be used from now on.

For $u\in H^{m+5}(\R)$ and $m\geq 0$, let
$$\Phi_{\sigma,m}(u) = \frac{1}{2}\kma{u}{\sigma}^2=\frac{1}{2}\sum\limits_{j=0}^m\frac{1}{(j!)^2}e^{2\sigma j} \sobolev{\partial_x^ju}{2}^2.$$
Given $u_0\in G^{1,s}(\R), s>5/2,$ such that $m_0$ does not change sign, let $u\in C([0,\infty),H^{\infty}(\R))$ be the unique global solution of \eqref{nonlocal1}. From the embedding $G^{1,s}(\R)\subset A(1),$ we have that $u_0\in A(1)$. Let $\sigma_0<0=: \bar{\sigma}$, which means that $e^{\sigma_0}<1$ and, from the definition of $A(1)$, we have $\Vert u\Vert_{\sigma_0,2}<\infty$.

For the global solution $u$, fix $T>0$ and define $\mu = 1+\max\{\sobolev{u}{2}; t\in [0,T]\}$ and $O= \{v\in H^{m+5}(\R); \Vert v\Vert_{H^2}<\mu\}$. Observe that the family $\{\Psi_{\sigma,m}; -\infty < \sigma < \bar{\sigma}, m\geq 0\}$ is well-defined on $O$ and $F:O\rightarrow X$ is continuous. Moreover, for this same family item $(a)$ is satisfied, see Kato and Masuda \cite{KM}, page 460. For item $(b)$, we will need the following result.

\begin{proposition}\label{prop5.5}
    Given $u\in H^{m+5}(\R), m\geq 0,$ for $\sigma\in \R$ we have the bound
    \begin{align*}
        \vert D\Phi_{\sigma,m}F(u)\vert \leq \bar{K}(\sobolev{u}{2})\Phi_{\sigma,m}(u) + \bar{\alpha}(\sobolev{u}{2},\Phi_{\sigma,m}(u))\partial_{\sigma}\Phi_{\sigma,m}(u),
    \end{align*}
    where $\bar{K}(p) = 224p$ and $\bar{\alpha}(p,q) = 832(1+p)q^{1/2}$.
\end{proposition}
\begin{proof}
    Since $F(u) = \displaystyle{-\frac{1}{2}\partial_x[u^2 + 3(1-\partial_x^2)^{-1}u_x^2]}$ and $\displaystyle{\frac{1}{2}D\sobolev{\partial_x^ju}{2}^2w = \langle \partial_x^jw\,,\, \partial_x^j u\rangle_{H^2}},$
    see \cite{BHP,KM}, by making $w = F(u)$ and summing over $j$ according to $\Phi_{\sigma,m}$ we have
    \begin{align*}
        \vert D\Phi_{\sigma,m}(u)F(u)\vert =&\left\vert \sum\limits_{j=0}^{m}\frac{e^{2\sigma j}}{(j!)^2}\langle\partial_x^j u\,,\,\partial_x^jF(u)\rangle_{H^2}\right\vert\\
        \leq& \left\vert\sum\limits_{j=0}^{m}\frac{e^{2\sigma j}}{(j!)^2}\langle\partial_x^j u\,,\,\partial_x^j(uu_x)\rangle_{H^2} \right\vert+ \frac{3}{2}\left\vert\sum\limits_{j=0}^{m}\frac{e^{2\sigma j}}{(j!)^2}\langle\partial_x^j u\,,\,\partial_x^{j+1}(1-\partial_x^2)^{-1}u_x^2\rangle_{H^2}\right\vert.
    \end{align*}
    From the proof of Lemma 4.1 in \cite{BHP1} (equations (6.14) and (6.16) with $k=1$), we know that
    $$\left\vert\sum\limits_{j=0}^{m}\frac{e^{2\sigma j}}{(j!)^2}\langle\partial_x^j u\,,\,\partial_x^j(uu_x)\rangle_{H^2} \right\vert \leq \bar{K}_1(\sobolev{u}{2})\Phi_{\sigma,m}(u) + \alpha_1(\sobolev{u}{2},\Phi_{\sigma,m}(u))\partial_{\sigma}\Phi_{\sigma,m}(u),$$
    where $\bar{K_1}(p) = 32p$ and $\alpha_1(p,q) = 64(1+p)q^{1/2}$, and
    $$\frac{1}{2}\left\vert\sum\limits_{j=0}^{m}\frac{e^{2\sigma j}}{(j!)^2}\langle\partial_x^j u\,,\,\partial_x^{j+1}(1-\partial_x^2)^{-1}u_x^2\rangle_{H^2}\right\vert \leq \bar{K}_2(\sobolev{u}{2})\Phi_{\sigma,m}(u) + \alpha_2(\sobolev{u}{2},\Phi_{\sigma,m}(u))\partial_{\sigma}\Phi_{\sigma,m}(u),$$
    where $\bar{K_2}(p) = 64p$ and $\alpha_1(p,q) = 256(1+p)q^{1/2}$. Under substitution of the respective terms in the inequality for $D\Phi_{\sigma,m}(u)F(u)$ we obtain
    \begin{align*}
        \vert D\Phi_{\sigma,m}(u)F(u)\vert \leq& \bar{K}(\sobolev{u}{2})\Phi_{\sigma,m}(u) + \bar{\alpha}(\sobolev{u}{2},\Phi_{\sigma,m}(u))\partial_{\sigma}\Phi_{\sigma,m}(u),
    \end{align*}
    where $\bar{K}(p) = 224p$ and $\bar{\alpha}(p,q) = 832(1+p)q^{1/2}$.
\end{proof}

\textbf{Proof of Proposition \ref{teo5.1}.} To prove the proposition, we basically need to complete the details for item $(b)$ of the Kato-Masuda Theorem. Therefore, we need to find $\bar{r}>0$ and continuous functions $\alpha(r)$ and $\beta(r)$ for $-\infty<r<\bar{r}$.

For $\bar{K}$ and $\bar{\alpha}$ given in Proposition \ref{prop5.5}, let 
\begin{align*}
    &K = \bar{K}(\mu),\quad \beta(r) = Kr,\quad r\geq 0,\\
    &\rho(t) = \frac{1}{2}\km{u_0}{\sigma_0}{2}^2e^{Kt},\quad \rho_m(t) = \frac{1}{2}\kma{u_0}{\sigma_0}^2e^{Kt},\\
    &\bar{r} = 1+ \max\{\rho(t);t\in[0,T]\}, \quad \alpha(r) = \bar{\alpha}(\mu,r),
\end{align*}
Observe that
\begin{enumerate}
    \item[(i)] $\alpha(r)$ and $\beta(r)$ are continuous for $r<\bar{r}$;
    \item[(ii)] $\rho_m(t)\leq \rho(t),$ for $t\in[0,T]$ and $\rho_m(t)\to \rho(t)$ uniformly.
    \item[(iii)] $\bar{K}(p)= 224p$ and $\bar{\alpha}(\bar{p},q) = 832(1+\bar{p})q^{1/2}$, for fixed $\bar{p}$, are nondecreasing.
\end{enumerate}

From the definition of $\mu$, $\mu>\Vert u\Vert_{H^2}$ and then from observation (iii) above, for all $v\in O$, we have
\begin{align*}
    &\bar{K}(\Vert u\Vert_{H^2}) \leq \bar{K}(\mu) = K,\\
    &\bar{\alpha}(\Vert v\Vert_{H^2},\Phi_{\sigma,m})\leq \bar{\alpha}(\mu,\Phi_{\sigma,m}) = \alpha(\Phi_{\sigma,m}(v)).
\end{align*}
The inequality of Proposition \ref{prop5.5} then yields
\begin{align*}
    \vert D \Phi_{\sigma,m}(v)\vert &\leq K \Phi_{\sigma,m}(v) + \alpha(\Phi_{\sigma,m}(v))\partial_{\sigma}\Phi_{\sigma,m}(v)\\
    &= \beta(\Phi_{\sigma,m}(v)) + \alpha(\Phi_{\sigma,m}(v))\partial_{\sigma}\Phi_{\sigma,m}(v)
\end{align*}
and item $(b)$ is finally satisfied.

For the same $T>0$, let $b:=\sigma_0<\bar{\sigma}$ and observe that $\rho_m(t)$ is a solution for the Cauchy problem
\begin{align*}
    \frac{d}{dt} \rho_m(t) = \beta(\rho_m(t)),\quad \rho_m(0) = \Phi_{\sigma_0,m}(u_0),\quad t\geq 0.
\end{align*}
Since
\begin{align*}
    \Phi_{\sigma_0,m}(u_0) = \frac{1}{2} \Vert u_0\Vert_{\sigma_0,2,m}^2 \leq \frac{1}{2}\Vert u_0\Vert^2_{\sigma_0,2},
\end{align*}
Kato-Masuda Theorem says that for
$$\sigma_m(t) = \sigma_0 - \int_0^t \alpha(\rho_m(\tau))d\tau,\quad t\in[0,T],$$
we have
\begin{align*}
    \Phi_{\sigma_m(t),m}(u) = \frac{1}{2}\Vert u\Vert_{\sigma_m(t),2,m}^2\leq \rho_m(t)
    \leq \rho(t) = \frac{1}{2}\Vert u_0\Vert^2_{\sigma_0,2}e^{Kt},
\end{align*}
for $t\in[0,T]$. By letting $m\to \infty$, we obtain
$$\Vert u\Vert_{\sigma(t),2}^2\leq \Vert u_0\Vert^2_{\sigma_0,2}e^{Kt}\leq \Vert u_0\Vert^2_{\sigma_0,2} e^{KT}<\infty$$
and $u(t)\in A(r_1)$ for $r_1 = e^{\sigma(t)}\leq e^{\sigma(T)}$ for $t\in[0,T]$. Once we have the expressions for $\alpha$ and $\rho$, we can estimate the radius of spatial analyticity $\sigma(t)$. In fact, we have
\begin{align}\label{5.2.2}
    \sigma(t) = \sigma_0 - \int_0^t\alpha(\rho(\tau))d\tau = \sigma_0 - A(e^{Bt}-1),
\end{align}
where $A=\frac{26\sqrt{2}}{7\mu}(1+\mu)\km{u_0}{\sigma_0}{2}$ and  $B=112\mu$. Observe that since $\mu\geq 1,$ we have $A\leq \frac{52\sqrt{2}}{7}\km{u_0}{\sigma_0}{2}=:L_1$. By letting $L_2:=B$, we have
$$r(t) = e^{\sigma(t)} \geq e^{\sigma_0+L_1}e^{-L_1e^{L_2t}} = L_3e^{-L_1e^{L_2t}},$$
where $L_3=r(0)e^{L_1}$.

\subsection{Proof of Proposition \ref{prop5.3}}

Given $u_0 \in G^{1,s}(\R),$ with $s>5/2$, let $u\in C([0,\infty);A(r_1))$ be the unique solution whose existence is guaranteed by Proposition \ref{teo5.1}. For the same initial data, from Corollary \ref{cor3.1} there are $\tilde{T}>0$ and a unique solution $\tilde{u}\in C^{\omega}([0,\tilde{T}(1-\delta);G^{\delta,s}(\R))$ for $\delta\in (0,1)$. Let $T = \frac{\tilde{T}}{2}(1-\delta)$, that is,
$\delta = \displaystyle{1-2\frac{T}{\tilde{T}} =: \delta (T)},$
and $\tilde{u}\in C^{\omega}([0,T];G^{\delta,s}(\R))\subset C^{\omega}([0,T];A(\delta(T)))$ once $G^{\delta,s}(\R) \subset A(\delta)$.

Since $A(r)\hookrightarrow H^{\infty}(\R)$ for $r>0$, then $\tilde{u}\in C^{\omega}([0,T];H^{\infty}(\R))\subset C([0,T];H^{\infty}(\R))$. From the uniqueness of the solution, we know that $u=\tilde{u}$ for $t\in[0,T]$, which means that $u\in C^{\omega}([0,T];A(\delta(T))),$ for $T = \frac{\tilde{T}}{2}(1-\delta)$ and $\delta(T)>0$.

\subsection{Proof of Proposition \ref{prop5.4}}

We will make use of the following result from \cite{BHP1} (Lemma 5.1).

\begin{lemma}\label{lemma5.4}
    Let $\delta>0$ and $m\geq 1$. Then $E_{\delta,m}(\R)$ is continously embedded in $A(\delta)$. Conversely, if $f\in A(r)$ for some $r>0$ then $f\in E_{\delta,m}(\R)$ for all $\delta<r/e$.
\end{lemma}

\textbf{Proof of Proposition \ref{prop5.4}.} We will prove the result by contradiction as we assume that $T^{\ast}<\infty$. For the solution $u \in C^{\omega}([0,T];A(\delta(T)))\subset C([0,T];H^{\infty}(\R))$, from the uniqueness of the global solution and the definition of $T^{\ast}$ it is immediate that if $T^{\ast}<\infty$, then $u(T^{\ast})$ is well-defined. Moreover, from Proposition \eqref{teo5.1}, we have $u(T^{\ast})\in A(r_1)$.

Let $\delta_0<\min \{1, r_1/e\}$ and then the converse of Lemma \ref{lemma5.4} tells that
$$u(T^{\ast})\in A(r_1)\subset E_{\delta_0,m}(\R),\quad m\geq3,$$
for all $\delta_0>0.$ From Proposition \ref{teo1}, there exist $\epsilon>0$ and a unique solution $\tilde{u}\in C^{\omega}([0,\epsilon];E_{\delta,m}(\R))$ for $0<\delta<\delta_0$ such that $\tilde{u}(0) = u(T^{\ast})$. On the other hand, since $E_{\delta,m}(\R)\hookrightarrow A(\delta)\subset H^{\infty}(\R),$ we have
$$\tilde{u}\in C^{\omega}([0,\epsilon];H^{\infty}(\R))\subset C([0,\epsilon];H^{\infty}(\R))$$
and the uniqueness of the global solution tells that $\tilde{u}(0) = u(T^{\ast})$, which means that
$$\tilde{u}(t) = u(T^{\ast}+t),\quad t\in[0,\epsilon].$$

Let $s=T^{\ast}+t$. Then $u(s) = \tilde{u}(s-T^{\ast}),$ for $s\in[T^{\ast},T^{\ast}+\epsilon],$
that is,
$$u \in C^{\omega}([T^{\ast},T^{\ast}+\epsilon];E_{\delta,m}(\R))\subset C^{\omega}([T^{\ast},T^{\ast}+\epsilon];A(\delta)).$$
Based on the definition of $T^{\ast}$, let $T>0$ be such that $T^{\ast}-\epsilon< T< T^{\ast}$ and the solution $u$ then belongs to $C^{\omega}([0,T];A(\delta(T)))$ for some $\delta(T)>0$.

Observe now that if $\sigma'\geq \sigma,$ then $A(\sigma')\subset A(\sigma).$ By defining $\tilde{\delta} = \min\{\delta,\delta(T)\}$, which means in particular that $\delta \geq \tilde{\delta}$ and $\delta(T)\geq \tilde{\delta}$, then
$$u\in C^{\omega}([0,T^{\ast}];A(\tilde{\delta})),\quad\text{and}\quad u\in C^{\omega}([T^{\ast},T^{\ast}+\epsilon];A(\tilde{\delta})),$$
which says that $T^{\ast}$ cannot be the supremum. As a result of the contradiction, $T^{\ast}$ must be infinite and, for every $T>0$, there exists $r(T)>0$ such that $u\in C^{\omega}([0,T];A(r(T)))$.

\section*{Acknowledgments}

This work was supported by FAPESP (grant number 2019/23688-4) and by the Royal Society under a Newton International Fellowship (reference number 201625).

\end{document}